\documentclass[11pt]{article}

\usepackage{amsfonts,amsmath, amssymb,latexsym}
\usepackage{multirow,comment}
\def\Z{{\mathbb Z}}
\def\A{{\mathbb A}}

\def\dist{{\rm dist}}

\newcommand{\cO}{{\mathcal O}}

\def\P{{\mathbb P}}

\def\prim{{\rm prim}}

\def\dim{{\rm dim}}

\def\tTr{{\rm Tr}}

\def\Span{{\rm Span}}

\def\R{{\mathbb R}}
\def\F{{\mathbb F}}

\def\Q{{\mathbb Q}}

\def\cB{{\mathcal B}}

\def\Z{{\mathbb Z}}

\def\P{{\mathbb P}}
\def\F{{\mathbb F}}
\def\Q{{\mathbb Q}}
\def\C{{\mathbb C}}

\def\W{{\mathcal W}}

\def\fz1{{F_{\Z,1}}}

\def\max{{\rm max}}

\newtheorem{theorem}{Theorem}
\newtheorem{corollary}[theorem]{Corollary}

\newtheorem{lemma}[theorem]{Lemma}

\newtheorem{remark}[theorem]{Remark}
\newtheorem{proposition}[theorem]{Proposition}

\newenvironment{proof}{\noindent {\bf Proof:}}{$\Box$ \vspace{2 ex}}

\usepackage{xcolor,color,url,lmodern}

\title{An improvement on Schmidt's bound on  the number of \\ number fields of  bounded discriminant and small degree}

\author{Manjul Bhargava, Arul Shankar, and Xiaoheng Wang}

\begin{document}
\maketitle

\begin{abstract}
We prove an improvement on Schmidt's upper bound on the number of number fields of degree~$n$ and absolute discriminant less than $X$ for $6\leq n\leq 94$. 
We carry this out by improving and applying a  uniform bound on the number of monic integer  polynomials, having bounded height and  discriminant divisible by a large square, that we proved in a previous work~\cite{SF1}.
\end{abstract}


\section{Introduction}

For $n\geq 2$, let $N_n(X)$ denote the number of isomorphism classes of number fields of degree $n$ having absolute discriminant less than $X$. In 1995, Schmidt~\cite{Sch}, proved the following upper bound on~$N_n(X)$:
\begin{theorem}[Schmidt]
We have
\begin{equation}\label{eqSch}
N_n(X)\ll_n X^{\frac{n+2}{4}}.
\end{equation}
\end{theorem}

A folklore conjecture predicts that $N_n(X)\asymp X$. This conjecture is elementary for $n=2$; for $n=3$, it was proven by Davenport and Heilbronn~\cite{DH}, and for $n=4,$ $5$  by the first-named author \cite{dodqf,dodpf}; these works in fact  determine asymptotics for $N_n(X)$ as $X\to\infty$. 

For large $n$, the exponent of $X$ in \eqref{eqSch} was  substantially improved by  Ellenberg--Venkatesh \cite{EV}, Couveignes \cite{Cou}, and Lemke Oliver--Thorne \cite{OT}, to $O(\exp(c\sqrt{\log n}))$, $O(\log^3 n)$, and $O(\log^2 n)$, respectively.
In~particular, in view of the implied $O$-constants, the latter work of Lemke Oliver and Thorne improved Schmidt's bound for all $n\geq 95$, while Schmidt's bound has remained the best known for $6\leq n\leq 94$. 

The aim of this paper is to improve the Schmidt bound for 
$6\leq n\leq 94$.
More precisely, we prove the following result.
\begin{theorem}\label{thimp}
For $n\geq6$, we have
\begin{equation*}
N_n(X)\ll_\epsilon X^{\frac{n+2}{4}-\frac{1}{2n-2}+\frac{1}{2^{2g}(2n-2)}+\epsilon},
\end{equation*}
where $g=\left\lfloor \frac{n-1}2\right\rfloor$.
\end{theorem}
For example, the best known bound for the number of isomorphism classes of sextic fields of absolute discriminant less than $X$ was previously $O(X^2)$, while  Theorem~\ref{thimp} yields the bound $O(X^{61/32+\epsilon})$.

\begin{remark}{\em
{\bf (a)} In an independent work, Anderson, Gafni, Hughes, Lemke Oliver,
Lowry-Duda, Thorne, Wang, and Zhang \cite{AGHLLTWZ} also obtain an improvement to Schmidt's bound, proving that $N_n(X)\ll_\epsilon X^{(n+2)/4-1/(4n-4)+\epsilon}$.

{\bf (b)} Improvements to Schmidt's bound in small degrees have played an important role in recent works, such as in proving van der Waerden's Conjecture on Galois groups in small degrees~\cite{vdw}; indeed, the improvements contained in Theorem~\ref{thimp} also immediately lead to corresponding improvements on the bounds in \cite{vdw}. 

{\bf (c)} Schmidt's original work obtains upper bounds for the number of degree $n$ extensions (with bounded absolute discriminant) of any fixed number field $K$. In forthcoming work \cite{Global2}, we generalize the results of \cite{SF1}, obtaining upper bounds for the number of monic polynomials with coefficients in the ring of integers of $K$, with bounded height and whose discriminants are divisible by the square of a large prime ideal. In combination with the methods of this paper, this will yield improved upper bounds on Schmidt's result for $K$.
}\end{remark}

\paragraph*{Methods.}
Our proof follows the strategy of \cite{Sch}. For a monic real polynomial \begin{equation}
\label{feq}    
f(x)=x^n+a_1x^{n-1}+a_2x^{n-2}+\cdots+a_n,
\end{equation} 
we define the {\it height} $H(f)$ of $f$ by 
\begin{equation*}
H(f):=\max(|a_i|^{1/i}).
\end{equation*}

Let $K$ be a number field of degree $n$ having absolute discriminant less than $X$. Then the lattice $\cO_K^{\tTr=0}$ (viewed as a subset of $\R^{n-1}$ via its archimedean embeddings) has covolume at most $O(\sqrt{X})$ in $\R^{n-1}$; hence 
the length of its shortest nonzero vector $\alpha$ is at most $O((\sqrt{X})^{1/(n-1)})=O(X^{1/(2n-2)})$. If $K$ is {\it primitive}, i.e., $K$ has no nontrivial  subfield, then $K=\Q(\alpha)$ 
since $\alpha\not\in\Q$. 
The minimal polynomial $f(x)$ of $\alpha$ is an integer monic  polynomial with height at most  $O(X^{1/(2n-2)})$ and vanishing subleading coefficient. That is, if $f(x)$ is expressed in the form (\ref{feq}), 
then $a_1=0$ and 
$|a_i|\ll X^{i/(2n-2)}$ for $i=2,\ldots,n$. Therefore, the total number of primitive number fields of degree~$n$ with absolute  discriminant less than $X$ is bounded by the number of such monic degree $n$  integer polynomials with height at most  $O(X^{1/(2n-2)})$ and vanishing subleading coefficient. The number of such polynomials is  $$\ll X^{\frac{2+3+\cdots + n}{2n-2}} = X^{\frac{n+2}{4}},$$
yielding Schmidt's bound for primitive number fields of degree $n$. Meanwhile, imprimitive fields of degree $n$ of absolute discriminant at most $X$ can be counted using their proper primitive subextensions (which would thus have degree at most $n/2$) to obtain the bound $O(X^{(n/2+2)/4})=O(X^{(n+4)/8})$, which is much smaller than Schmidt's bound for primitive fields.

Our improvement on the above argument in the case of primitive fields is based on the following observation: a monic integral polynomial $f(x)$ of degree $n$ with height $\asymp X^{1/(2n-2)}$ usually has discriminant close to $X^{n/2}$; in this scenario, if the number field $K=\Q[x]/(f(x))$ has absolute discriminant less than $X$, then the discriminant of $f(x)$ must be divisible by a large square! This is because the discriminant of $f$ is equal to the discriminant of $K$ times the square of the index of $\Z[x]/f(x)$ in the ring of integers of $K$. Thus,  to bound the number of  primitive number fields $K$ of degree $n$ having absolute discriminant less than $X$, it~suffices to bound the number of monic polynomials of height $\ll X^{1/(2n-2)}$ with vanishing subleading coefficient whose discriminant is divisible by a large square. 

In previous work \cite{SF1}, in order to sieve to squarefree discriminants, we proved an upper bound on the number of monic integral polynomials of degree $n$ and bounded height having discriminant  divisible by the square of a large squarefree number; this was accomplished via a lift to pairs of integral $n$-ary quadratic forms $(A_0,B)$ where $A_0$ is split.
Specifically, for a positive integer $m$, 
let $\W_m$ denote the set of integer monic polynomials of degree $n$ having discriminant divisible by $m^2$. Then, in  \cite[Theorem 4.4]{SF1}, we proved the following theorem:

\begin{theorem}\label{thSF1}
For a real number $H>1$, we have
\begin{equation*}
\#\bigcup_{\substack{m>M\\m\textrm{\em \:squarefree}}}\bigl\{f(x)\in \W_m:H(f)<H\bigr\}
= O_\epsilon\Bigl(\frac{H^{n(n+1)/2+\epsilon}}{\sqrt{M}}\Bigr)+O_\epsilon\bigl(H^{n(n+1)/2-1/5+\epsilon}\bigr).
\end{equation*}
\end{theorem}

In this paper, we improve upon Theorem~\ref{thSF1} in two ways. First, we generalize these  results to bound the number of monic integer polynomials of degree $n$ having discriminant divisible by the square of a large  (not necessarily squarefree)  integer. Second, we replace the use of the Selberg sieve in \cite{SF1} by a Hilbert irreducibility theorem (HIT) argument---using the quantitative version of HIT due to Castillo and Dietmann~\cite{CD}---to attain an improved error term.  We thereby prove the following theorem:

\begin{theorem}\label{thMainPolybound}
For a real number $H>1$, we have
\begin{equation*}
\begin{array}{rcl}
\displaystyle
\#\bigcup_{\substack{m>M}}\#\bigl\{f(x)\in \W_m:H(f)<H\bigr\}
&=& \displaystyle O_\epsilon\Bigl(\frac{H^{n(n+1)/2+\epsilon}}{M^{2/(n+3)-\epsilon}}\Bigr)+O\bigl(H^{n(n+1)/2-1+1/2^{2g}+\epsilon}\bigr);
\\[.25in]
\displaystyle
\#\bigcup_{\substack{m>M\\m\textrm{\em \:squarefree}}}\#\bigl\{f(x)\in \W_m:H(f)<H\bigr\}
&=& \displaystyle O_\epsilon\Bigl(\frac{H^{n(n+1)/2+\epsilon}}{\sqrt{M}}\Bigr)+O\bigl(H^{n(n+1)/2-1+1/2^{2g}+\epsilon}\bigr).
\end{array}
\end{equation*}
\end{theorem}
This improvement to Theorem~\ref{thSF1} then allows us to deduce Theorem~\ref{thimp}. 

We note that Theorem~\ref{thMainPolybound} also implies corresponding improvements to the error terms in the main results of  \cite{SF1}. Let $V_n\simeq\A^n$ denote the space of monic polynomials of degree $n$. An element $f\in V_n(\Z)$ is said to be {\it maximal} if it has nonzero discriminant and if $\Z[x]/f(x)$ is the ring of integers of $\Q[x]/f(x)$. Then, for certain  constants $\lambda_n>0$ as defined in \cite[Equation (1)]{SF1}, we prove:

\begin{theorem}\label{thm:mainterm} We have
\begin{equation*}
\begin{array}{rcl}
\displaystyle \#\bigl\{f\in V_n(\Z): H(f)<H \textrm{\em \;and } \Delta(f) 
\textrm{\em \;squarefree}\bigr\} &=& \displaystyle
\lambda_n 2^n H^{\frac{n(n+1)}{2}}+O_\epsilon\Bigl(
H^{\frac{n(n+1)}{2}-1+\frac{1}{2^{2g}}+\epsilon} \Bigr);
\\[.2in]
\displaystyle \#\bigl\{f\in V_n(\Z): H(f)<H \textrm{\em \;and } f 
\textrm{\em \;maximal}\bigr\} &=& \displaystyle
\frac{6}{\pi^2} 2^n H^{\frac{n(n+1)}{2}}+O_\epsilon\Bigl(
H^{\frac{n(n+1)}{2}-1+\frac{1}{2^{2g}}+\epsilon} \Bigr).
\end{array}
\end{equation*}
\end{theorem}
This improves upon  \cite[(4)]{SF1}, where the error terms were $O\bigl(
H^{\frac{n(n+1)}{2}-\frac15+\epsilon}\bigr)$.

One direct consequence of this improvement is improved {\it level-of-distribution} results when counting monic polynomials $f(x)$
such that $\Delta(f)$ is squarefree (resp.\ $f(x)$ is maximal) that satisfy splitting conditions modulo finitely many primes. A concrete application concerns the distribution of low-lying zeroes of the Dedekind zeta functions of monogenized number fields of degree~$n$. (A pair $(K,\alpha)$, where $K$ is a number field and $\alpha$ is an element in the ring of integers of $K$, is said to be {\it monogenized} if $\Z[\alpha]$ is the ring of integers of $K$.) In \cite[\S5]{SST}, it was shown that this family of zeta functions has symplectic symmetry type, via a computation of the $1$-level density with respect to test functions whose Fourier transforms have bounded support in $[-\alpha,\alpha]$ for $\alpha<2/(5n(n+1)(2n+1))$. Theorem~\ref{thMainPolybound} implies that we can in fact take $\alpha<2/((1-2^{-2g})n(n+1)(2n+1))$. Improved level-of-distribution results in this setting also have applications towards proving the existence of number fields of degree $n$ whose discriminants have a bounded number of prime factors; see \cite{TT}, where such results are proved for cubic and quartic 
fields.

\paragraph*{Organization.} This paper is organized as follows.  In \S2, we bound the number of monic integer polynomials of degree $n$ and height less than $H$ that have discriminant smaller than the expected size $\asymp H^{n(n-1)}$. In \S3, we then consider the set of monic integer polynomials of degree $n$ with discriminant  divisible by the square of a large, not-necessarily-squarefree  integer. We partition this set into two subsets whose sizes we effectively bound using a geometric sieve argument (as in \cite{geosieve}) and via a lift to pairs of quadratic forms $(A_0,B)$ (as in \cite{SF1}), respectively. In \S4, we describe how to replace the Selberg sieve argument of \cite{SF1} by a quantitative Hilbert irreducibility argument using the work of Castillo and Dietmann~\cite{CD} to improve error terms. 
Finally, in \S5, we combine the results of \S2, \S3, and \S4 to prove Theorems~\ref{thMainPolybound} and  \ref{thimp} and \ref{thm:mainterm}.

\section{The number of monic polynomials of bounded height and small discriminant}

The following proposition shows that most monic integer polynomials of degree $n$ with vanishing subleading coefficient and height less than $H$ have discriminant close to $H^{n(n-1)}$.

\begin{proposition}\label{thm:discsmall}
Let $0<\kappa<n(n-1)$. The number of monic integer polynomials having degree $n$, vanishing subleading coefficient,  height less than $H$, and absolute discriminant less than $H^{n(n-1)-\kappa}$ is at most  $O(H^{(n-1)(n+2)/2 - {\kappa}/{(n-1)}})$, where the implied constant depends only on $n$.
\end{proposition}

\begin{proof}
Let $\eta = \kappa/(n-1)$.
Let $a_2,\ldots,a_{n-1}$ be integers with $|a_i|<H^i$ for $i=2,\ldots,n-1$. The discriminant of $x^n + a_2x^{n-2} + \cdots + a_n$ is a polynomial $F(a_n)$ in $a_n$ of degree $n-1$ with leading term $C_n a_n^{n-1}$ where $C_n$ is a nonzero constant. Let $s_1,\ldots,s_{n-1}\in\C$ be the $n-1$ roots of $F(x)$. Then $$F(a_n) = C_n(a_n-s_1)\cdots(a_n-s_{n-1}).$$
Since $|F(a_n)| \leq H^{n(n-1)-\kappa}$, it follows that $|(a_n-s_1)\cdots(a_n-s_{n-1})|\ll_n H^{n(n-1)-\kappa}$. Hence 
\begin{equation}\label{aneq}
|a_n - s_i| \ll_n H^{n-\eta} 
\end{equation}
for some $i\in\{1,\ldots,n-1\}$.
The number of  integers $a_n$ satisfying (\ref{aneq}) for some $i$ is $O(H^{n-\eta}+1)$. Multiplying this by the number of choices for $a_2,\ldots,a_{n-1}$ then gives the desired bound.
\end{proof}

\section{The divisibility of discriminants of polynomials by large squares}

In this section, we consider the set $\W_m$ of monic integer polynomials of degree $n$ having discriminant  divisible by $m^2$. First, for a prime power $m=p^k$, we write the set $\W_{p^k}$ of polynomials naturally as a union of two sets. The first set  consists of polynomials $f\in \W_{p^k}$ satisfying  $p^k\mid\gcd(\Delta(f),\Delta'(f))$ for some specifically constructed polynomial  $\Delta'(f)$ in the coefficients of $f$. The first set is designed so that a suitable extension of the techniques of \cite{geosieve} can be applied to bound the number of elements of bounded height in this first~set. The second set  consists of polynomials $f\in \W_{p^k}$ for which there exists some $r\in\Z$ such that $f(r)$ and $f'(r)$ are both divisible by a  high power of $p$. The second set is designed so that a suitable adaptation of the methods of \cite{SF1} can be applied to bound the number of elements of bounded height in this second set. Finally, by using the case of prime powers~$m=p^k$, we show that $\W_m$, for a general positive integer $m$, can also be expressed as the union of two sets, on which suitable extensions of the methods from \cite{geosieve} and \cite{SF1}, respectively, can be applied to bound the number of elements in $\W_m$ of bounded height.

We begin by defining $\Delta'$.
For any monic  polynomial $f(x) = x^n + a_1x^{n-1} + \cdots + a_n$ of degree~$n$, its discriminant $\Delta(f)=\Delta(a_1,\ldots,a_n)$ can be viewed as a polynomial in $a_1,\ldots,a_n$ with integer coefficients. Let $r_1,\ldots,r_n$ denote the $n$ roots of $f(x)$. Define
$$\Delta'(f) = \sum_{i<j} \frac{\Delta(f)}{(r_i-r_j)^2}.$$
Note that we may represent $\Delta'(f)$ as a polynomial $\Delta'(a_1,\ldots,a_n)$ in $a_1,\ldots,a_n$ with integer coefficients, since it is a symmetric polynomial in $r_1,\ldots,r_n$ with integer coefficients.

Recall that $V_n\simeq\A^n$ denotes the space of monic polynomials of degree $n$. For a positive integer $m$, define the sets
\begin{eqnarray*}
    \W_m &=& \{f\in V_n(\Z) \colon m^2\mid \Delta(f),\,\,\Delta(f)\neq 0\};\\[.05in]
    \W_{m}^{(1)} &=& \{f\in \W_m \colon m\mid \Delta'(f)\};\\[.05in]
    \W_{m}^{(2)} &=& \{f\in \W_m \colon \exists r\in\Z \mbox{ such that }m\mid f'(r),\;m^2\mid f(r)\}.
\end{eqnarray*}
We begin with the following result on $\W_{p^k}$ for a prime power $p^k$.
\begin{lemma}\label{thm:sw}
Let $p$ be an odd prime and $k$ be any positive integer. Then
\begin{equation*}
\begin{array}{rcl}
\displaystyle \W_{p^k} &\subset& 
\displaystyle \W_{p^k}^{(1)} \cup \W_{p^{\lceil k/2\rceil}}^{(2)};
\\[.1in]\displaystyle
\displaystyle \W_{2^k} &\subset& 
\displaystyle \W_{2^k}^{(1)} \cup \W_{2^{\lceil k/2\rceil-1}}^{(2)}.
\end{array}
\end{equation*}
\end{lemma}

\begin{proof}
Let $v_p$ denote the $p$-adic valuation.
Suppose $f(x)\in \W_{p^k}$ and let $\ell= v_p(\Delta(f))\geq 2k$. Let $r_1,\ldots,r_n$ denote the $n$ roots of $f(x)$ in $\overline{\Q_p}$. Suppose the $p$-adic valuation $\theta  = v_p(r_1-r_2)$ is the largest among all differences of roots of $f(x).$ Then $p^{\ell-2\theta }\mid\Delta'(f)$. If $\theta \leq (\ell-k)/2$, then $f\in \W_{p^k}^{(1)}$.

Now suppose that $\theta  > (\ell-k)/2\geq k/2$. Then $v_p(r_i-r_j)<\theta$ for all distinct pairs $i,j$ with $\{i,j\}\neq \{1,2\}$;
indeed, if  $v_p(r_i-r_j)=\theta$,
then $(r_1-r_2)^2(r_i-r_j)^2\mid\Delta(f)$ implying $\ell\geq 4\theta >2\ell-2k$, a contradiction. Thus either $r_1,r_2$ are defined over $\Q_p$ or are conjugate over some quadratic extension of $\Q_p$. That is, we have $q(x) = (x-r_1)(x-r_2) = x^2 - bx + c$ for some $b,c\in\Z_p$. We claim that $b/2\in \Z_p$. This is clear for $p\neq 2$, while if $p=2$, then $b^2 - 4c = (r_1-r_2)^2$ is divisible by~$2^{2\theta }$ and hence is divisible by $2^2$ since $\theta >k/2\geq 1/2$.

Now $q(b/2) = - \frac14(r_1-r_2)^2$ and $q'(b/2) = 0$. Since $f(x) = q(x)h(x)$ for some $h(x)\in\Z_p[x]$, we have that $\frac14(r_1-r_2)^2$ divides $f(b/2)$ and $f'(b/2)$. Hence $f\in \W_{p^e}^{(2)}$ where $e= \lceil \theta \rceil - v_p(2)$.
\end{proof}

We now prove the main result of this section.
\begin{proposition}\label{cor:sw}
Let $M$ be any positive real number and let $q_1 = \lfloor (M/2)^{\alpha}\rfloor$, $q_2 = \lfloor (M/2)^{\beta}\rfloor$, where $\alpha,\beta$ are positive real numbers such that $\alpha + 2\beta = 1$. Then 
\begin{equation*}
\begin{array}{rcl}
\displaystyle\bigcup_{\substack{m>M\\m\textrm{\em \:squarefree}}}\W_m &\subset& \displaystyle\bigcup_{m>\sqrt{M}}\W_{m}^{(1)}\cup \bigcup_{m>\sqrt{M}}\W_{m}^{(2)};
\\[.375in]
\displaystyle\bigcup_{m>M}\W_m &\subset& \displaystyle\bigcup_{m>q_1}\W_{m}^{(1)}\cup \bigcup_{m>q_2}\W_{m}^{(2)}.
\end{array}
\end{equation*}
\end{proposition}

\begin{proof}
The first containment is proved in the proof of \cite[Theorem 4.4]{SF1}. To obtain the second containment,
fix $f\in \W_m$ for some $m>M$. Let $\prod_{i=1}^n p_i^{k_i}$ be the prime factorization of $m$ and let
$$m_1 = \prod_{p_i\colon f\in \W_{p_i^{k_i}}^{(1)}} p_i^{k_i},\qquad m_2 = \prod_{p_i\colon f\in \W_{p_i^{\lceil k_i/2\rceil-v_{p_i}(2)}}^{(2)}} p_i^{\lceil k_i/2\rceil-v_{p_i}(2)}.$$
Then $f\in \W_{m_1}^{(1)}\cup \W_{m_2}^{(2)}$ and $m_1m_2^2\geq m/2 > M/2$. Hence either $m_1>q_1$ or $m_2>q_2$, and therefore $f\in \bigcup_{m>q_1}\W_{m}^{(1)}\cup \bigcup_{m>q_2}\W_{m}^{(2)}$, as desired.
\end{proof}

\section{A rational root criterion for being distinguished}

Let $A_0$ denote the $n\times n$ symmetric matrix with $1$'s on the anti-diagonal and $0$'s elsewhere. Let~$W$ denote the space of $n\times n$ symmetric matrices. We recall  from \cite{BG,SW} the following definition of a distinguished element $B\in W(\Q)$. 
If $n=2g+1$ is odd, then an element $B\in W(\Q)$ with $\Delta(B)\neq 0$ is {\it distinguished} if there exists a $g$-plane $Y$ defined over $\Q$ that is isotropic with respect to the quadratic forms defined by $A_0$ and $B$. When $n = 2g+2$ is even, then $B$ is {\it distinguished} if there exists a $g$-plane $Y$ defined over $\Q$ such that $\Span(Y,TY)$ is isotropic with respect to the quadratic form defined by $A_0$, where $T = A_0^{-1}B$. 

We use $b_{11},b_{12},\ldots,b_{nn}$ to denote the (indeterminate)  entries of $B$. The goal of this section is to construct a polynomial $F(B,x)=F(b_{11},b_{12},\ldots,b_{nn},x)$ that is irreducible generically but has a rational root in $x$ if $B$ is distinguished. The construction of this polynomial will allow us to use a quantitative Hilbert irreducibility  argument \cite{CD} to bound the number of distinguished elements in bounded regions.

\begin{theorem}\label{thm:poly}
Suppose $n = 2g+1$ is odd. There exists a polynomial $F(B,x)\in \Z[\sqrt{-1}][b_{11},\ldots,b_{nn}][x]$ of degree $2^{2g}$ that is irreducible in $\Q[\sqrt{-1}](b_{11},\ldots,b_{nn})[x]$ such that for any $B_0\in W(\Q)$, if $B_0$ is distinguished, then $F(B_0,x)\in\Q[\sqrt{-1}][x]$ has a root in $\Q$.
\end{theorem}

\begin{proof}
Let $B\in W$ be the generic element. Fix $P\in M_n(\Q[\sqrt{-1}])$ such that $PA_0P^t$ is the identity matrix $I_n$. Let $B' = PBP^t$. Then the entries of $B'$ lie in $\Z[\sqrt{-1}][b_{11},\ldots,b_{nn}]$. Let $c_1,\ldots,c_n \in \overline{\Q(b_{11},\ldots,b_{nn})}$ be the eigenvalues of $B'$, and let $h(c_1,\ldots,c_n)\in \text{O}_n(\Q[\sqrt{-1}](b_{11},\ldots,b_{nn},c_1,\ldots,c_n))$ be a change-of-basis matrix such that $h(c_1,\ldots,c_n)B'h(c_1,\ldots,c_n)^t = B'' = \text{diag}(c_1,\ldots,c_n)$. Note that we have fixed an order of the eigenvalues here. For any other order, we simply multiply $h(c_1,\ldots,c_n)$ by the corresponding permutation matrix. 

Following \cite[\S2]{XWthesis}, we now have the following explicit construction of the $2^{2g}$ common isotropic $g$-planes with respect to the quadratic forms defined by $I_n$ and $B''$.
Let $D_1,\ldots,D_{n} \in \Q(c_1,\ldots,c_n)$ be a nonzero solution to the following system of linear equations:
\begin{eqnarray*}
D_1 + D_2 + \cdots + D_n &=& 0\\
D_1c_1 + D_2c_2 + \cdots + D_nc_n &=& 0\\
&\vdots& \\
D_1c_1^{2g-1} + D_2c_2^{2g-1} + \cdots + D_nc_n^{2g-1} &=&0.
\end{eqnarray*}
Note that we may take
\begin{equation}\label{eq:formD}
    D_i = \pm \prod_{j\neq i}(c_j - c_i)^{-1}.
\end{equation}
This is obtained by noting that the kernel of an incomplete Vandermonde matrix is spanned by the last row of the inverse of the completed Vandermonde matrix.
None of the $D_i$ is equal to $0$ and so for each choice of $d_i\in \Q(\sqrt{D_i})$ with $d_i^2 = D_i$, we have a $g$-plane
\begin{equation}\label{eq:formY}
    Y = \Span\{(d_1,\ldots,d_n), (d_1c_1,\ldots,d_nc_n),\ldots, (d_1c_1^{g-1},\ldots,d_nc_n^{g-1})\},
\end{equation} 
which is isotropic with respect to the quadratic forms defined by $I_n$ and $B''$.
Negating all of the $d_i$'s gives the same $Y$, and so we have $2^{2g}$ distinct $g$-planes.
Now $$P^th(c_1,\ldots,c_n)^tY=\Span\{(\ell_{11},\ldots,\ell_{1n}),\ldots,(\ell_{g1},\ldots,\ell_{gn})\}$$
gives a $g$-plane that is isotropic with respect to the quadratic forms defined $A_0$ and $B$, where each $\ell_{ij}$ is a linear form in $d_1,\ldots,d_n$ with coefficients in $\Q[\sqrt{-1}](b_{11},\ldots,b_{nn},c_1,\ldots,c_n)$.

Next we apply the Pl\"{u}cker embedding to send each $P^th(c_1,\ldots,c_n)^tY$ to a point $$[G_0(d_1,\ldots,d_n)\colon\cdots\colon G_N(d_1,\ldots,d_n)]$$ in projective space, where each $G_i$ is a homogeneous polynomial of degree $g$ with coefficients in $\Q[\sqrt{-1}](b_{11},\ldots,b_{nn},c_1,\ldots,c_n).$ As $d_1,\ldots,d_{n}$ vary, we get $2^{2g}$ points $P_1,\ldots,P_{2^{2g}}$ this way. Now let $L_1(x_0,\ldots,x_N) = \alpha_0x_0 + \cdots + \alpha_Nx_N$ and $L_2(x_0,\ldots,x_N) = \beta_0x_0 + \cdots + \beta_Nx_N$ be two linear forms with integer coefficients $\alpha_i,\beta_i$ to be chosen later. For $i = 1,\ldots,2^{2g}$, let $$Q_i = [L_1(P_i)\colon L_2(P_i)]\in\P^1(\Q[\sqrt{-1}](b_{11},\ldots,b_{nn},c_1,\ldots,c_n)(d_1,\ldots,d_n)).$$ For now, we only require that none of the $Q_i$ equals $[1:0]$. We then see that there is a binary $2^{2g}$-ic form $J(c_1,\ldots,c_n)(x,y)$ defined over $\Q[\sqrt{-1}](b_{11},\ldots,b_{nn},c_1,\ldots,c_n)$ vanishing on $Q_1,\ldots,Q_{2^{2g}}$. By scaling, we may assume that $$J(c_1,\ldots,c_n)(x,y)\in\Z[\sqrt{-1}][b_{11},\ldots,b_{nn},c_1,\ldots,c_n][x,y].$$ Let $J_1(c_1,\ldots,c_n)(x) = J(c_1,\ldots,c_n)(x,1)$. Note that if $B_0\in W(\Q)$ is distinguished, then one of the $P_i$ is defined over $\Q$, in which case $J_1(c_1,\ldots,c_n)(x)$ has a root over $\Q$.

Finally, we note that the homogeneous polynomials $G_0,\ldots,G_N$ are independent of the ordering of the eigenvalues $c_1,\ldots,c_n$, since permuting the $c_i$'s permutes the coordinates of $Y$, which is then cancelled by the extra permutation matrix in $h(c_1,\ldots,c_n)$. Therefore, the coefficients of $J_1(c_1,\ldots,c_n)(x)$ are symmetric in $c_1,\ldots,c_n$ and so are polynomials in the coefficients of the characteristic polynomial of $B$. We let $F(B,x)$ denote this polynomial.

It remains to prove that for some choice of coefficients $\alpha_i,\beta_i$ for the linear forms $L_1,L_2$, the polynomial $F(B,x)\in \Z[\sqrt{-1}][b_{11},\ldots,b_{nn}][x]$ is irreducible. It suffices to exhibit some $B_0\in W(\Q[\sqrt{-1}])$ such that $G_{\Q[\sqrt{-1}]}$, the absolute Galois group of $\Q[\sqrt{-1}]$, acts transitively on the $2^{2g}$ distinct roots of $F(B_0,x)$. For any $c_1,\ldots,c_n$, let $E_i=(c_1-c_i)\cdots (c_{i-1}-c_i)(c_{i+1}-c_i)\cdots(c_n-c_i) = \pm D_i^{-1}$.
Let $M$ be a large integer so that there are at least $n-1$ distinct primes $q_1,\ldots,q_{n-1}$ lying inside $(M - \sqrt{M}, M)$. Let $c_n = M$ and let $c_i = M - q_i$ for $i = 1,\ldots,n-1$. Then for any $i,j=1,\ldots,n-1$ with $j\neq i$, we have $q_i\mid E_i$ and $q_i\nmid E_j$. Note also that $E_n = (-1)^{n-1}q_1\cdots q_{n-1}$. For each $i=1,\ldots,n-1$, let $\sigma_i$ denote an element in the absolute Galois group $G_{\Q[\sqrt{-1}]}$ that negates $\sqrt{q_j}$ for all $j=1,\ldots,n-1$ for which $j\neq i$ and fixes all other square roots that appear (including $\sqrt{q_i}$). Since $n$ is odd, we see that $\sigma_i(d_j)=-d_j$ for all $j\neq i$ and $\sigma_i(d_i) = d_i$. Hence $$\sigma_i([d_1\colon\cdots\colon d_n]) =[d_1\colon\cdots\colon d_{i-1}\colon -d_i\colon d_{i+1}\colon\cdots\colon d_n].$$ That is, the absolute Galois group $G_{\Q[\sqrt{-1}]}$ acts transitively on the set $\{[d_1\colon\cdots\colon d_n]\mid d_i^2=D_i\}$. Hence, it also acts transitively on the $2^{2g}$ $Y$'s defined in \eqref{eq:formY} as the $d_i$'s vary, and so also acts transitively on $P_1,\ldots,P_{2^{2g}}$. We may simply choose the integers $\alpha_i,\beta_i$ so that $Q_1,\ldots,Q_{2^{2g}}$ are distinct. Let $B'=\text{diag}(c_1,\ldots,c_n)$ and let $B_0 = P^{-1}B'(P^{-1})^t\in W(\Q[\sqrt{-1}])$. Then $G_{\Q[\sqrt{-1}]}$ acts transitively on the $2^{2g}$ distinct roots of $F(B_0,x)$.
\end{proof}

\begin{theorem}\label{thm:even}
Suppose $n = 2g+2$ is even. Let $R = \Z[\sqrt{-1}][b_{11},\ldots,b_{nn}][c]/(\det(cA_0 - B))$ and let $K$ be its fraction field. Then there exists a polynomial $F(B,c,x)\in R[x]$ of degree $2^{2g}$ that is irreducible in $K[x]$ such that for any $B_0\in W(\Q)$, if $B_0$ is distinguished and $c_0\in\bar{\Q}$ is any eigenvalue of $B_0$, then  $F(B_0,c_0,x)\in\Q[\sqrt{-1}][c_0][x]$ has a root in $\Q$.
\end{theorem}

\begin{proof}
Given $(B,c)$, we proceed as before to orthogonally diagonalize $B'=PBP^t$ into $B'' = \text{diag}(c_1,\ldots,c_{n-1},c)$.  Then we take $B''' = B'' - cI_n = \text{diag}(c_1-c,\ldots,c_{n-1}-c,0)$ and use the construction above in the odd degree case with $c_1,\ldots,c_n$ replaced by $c_1-c,\ldots,c_{n-1}-c$ to obtain a polynomial whose coefficients are polynomials that are symmetric in $c_1 - c, \ldots, c_{n-1}-c$. Note that $(x-c_1)\cdots(x-c_{n-1}) = \det(xA_0-B)/(x-c)\in R[x]$. Hence $(x-(c_1-c))\cdots(x-(c_{n-1}-c))\in R[x]$ and so any polynomial that is symmetric in $c_1 - c, \ldots, c_{n-1}-c$ belongs to $R$.
\end{proof}

As an immediate consequence of Theorems \ref{thm:poly} and \ref{thm:even}, we may apply the quantitative Hilbert  irreducibility theorem of Castillo and Dietmann~\cite{CD} to bound the number of distinguished elements in homogeneously expanding sets. 
Let $W(\Z)^\dist$ denote the set of distinguished elements in $W(\Z)$. Then we have the following result.
\begin{corollary}\label{corhilbert}
Let $\cB\subset W(\R)$ be a bounded open set. Let $Y=(Y_{ij})$
be an $n\times n$ matrix of positive real numbers. 
Let $Y\cdot \cB$ be the set obtained by scaling the $(i,j)$-entries of elements in $\cB$ by $Y_{ij}$.
Then we have
\begin{eqnarray}\label{eqcor10}
\#\bigl\{Y\cdot \cB\cap W(\Z)^\dist\bigr\}&=& O_\epsilon\left(
\frac{\prod Y_{ij}^{1+\epsilon}}{\min\bigl\{Y_{ij}^{1-1/2^{2g}}\bigr\}}
\right).
\end{eqnarray}
\end{corollary}
\begin{proof}
Define $F_1(B,x)$ to be $N_{\Q[\sqrt{-1}]/\Q}F(B,x)$ in the odd case and to be $N_{K/\Q(b_{11},\ldots,b_{nn})}F(B,c,x)$ in the even case. Theorems \ref{thm:poly} and \ref{thm:even} imply that if $B_0\in W(\Q)$ is distinguished, then the Galois group of $F_1(B_0,x)$ has index at least $2^{2g}$ in the generic Galois group of $F_1(B,x)$. More precisely, they imply that every irreducible factor of $F_1(B,x)$ has degree at least $2^{2g}$, and that if $B_0$ is distinguished then $F_1(B_0,x)$ has a rational root.

We wish to upper bound the number of $B_0\in W(\Z)$, with $b_{ij}<Y_{ij}$, such that $f(B_0,x)$ has a rational root in $x$ for some irreducible factor $f(B,x)$ of $F_1(B,x)$. When all the $Y_{ij}$'s are the same, say $Y$, the required bound follows immediately from \cite[Theorem~1]{CD}. Indeed, when applied to  $f(B,x)$ with generic Galois group $G_f$, \cite[Theorem~1]{CD} states that the number of $B_0\in W(\Z)$ with each $|b_{ij}|<Y$, such that the Galois group of $f(B_0,x)$ is $K\subset G_f$, is $O(Y^{\dim(W)-1+|G_f/K|^{-1}+\epsilon})$. Since $B_0$ being distinguished implies that $f(B_0,x)$ has a  rational root for some irreducible factor $f(B,x)$ of $F(B,x)$, and the degree of $f(B,x)$ is at least $2^{2g}$, we have $|G_f/K|\geq 2^{2g}$ and the result follows.

The proof of \cite[Theorem~1]{CD} in the case where we require $f(B_0,x)$ to have a rational root (as opposed to a general Galois subgroup $K\subset G_f$) is much simpler than the general case (as the construction of a polynomial, associated to $f(B_0,x)$, having a rational root may be skipped in this case). We~now describe how this proof also carries through without any change for general $Y_{ij}$.
First, $f(B,x)$ can be assumed to be monic in $x$ by replacing $f(B,x)=g_0(B)x^m+g_1(B)x^{m-1}+\cdots+ g_m(x)$ by~$g_0(B)^{m-1}f(B,x/g_0(B))$. In order to make this reduction, it is necessary to provide an upper bound for the number of $B_0\in W(\Z)$ with $|b_{ij}|<Y_{ij}$ such that $g_0(B_0)=0$. This number is clearly bounded by the right hand side of \eqref{eqcor10}, as a fibering argument readily shows. Specifically, we fiber over all but one of the coefficients, denoted by $b$. Fixing values for each $b_{ij}\neq b$ yields a polynomial $g(b)=g_0(B)$ in one variable. If $g(b)$ is not identically zero, then $g(b)=0$ has $O(1)$ different solutions, yielding a sufficient saving. Meanwhile, the condition of $g(b)$ being identically zero imposes one or more polynomial vanishing conditions on the coefficients $b_{ij}\neq b$, and the number of such values of $b_{ij}$ also satisfies the required bound by induction. 

The result for monic polynomials is proved in \cite[Lemma~7]{CD} by fibering over all but one of the coefficients and then using induction. This proof (for the case when all the $Y_{ij}$'s are the same) carries over without change for general $Y_{ij}$, when the fibering is done over the variables $b_{ij}$ for which the $Y_{ij}$ are the smallest.
\end{proof}

\section{Proof of the main result}
In this section, we prove Theorem \ref{thMainPolybound}, and its analogue for monic integer polynomials having vanishing subleading coefficient. We then prove Theorem \ref{thm:mainterm} and Theorem \ref{thimp}.

We begin by bounding the number of monic integer polynomials with bounded height belonging to $\W_m^{(1)}$ (resp.\ $\W_m^{(2)})$ for some large $m$. To bound the number of elements in $\cup_{m>M} \W_{m}^{(1)}$, we have the following result:
\begin{proposition}\label{propW1}
We have
\begin{equation}\label{eqeq1}
\#\bigcup_{\substack{m>M
 }}\{f\in \W_{m}^{(1)}:H(f)<H\}=
O_\epsilon\Big(\frac{H^{n(n+1)/2+\epsilon}}{M^{2/(n-1)-\epsilon}}\Big)+O(H^{n(n+1)/2-1}).
\end{equation}
\end{proposition}
\begin{proof}
We begin by noting that the proofs of \cite[Theorem 3.5 and Lemma 3.6]{geosieve} imply the bound
\begin{equation}\label{eqeq2temp}
\#\bigcup_{\substack{m>M
 }}|\mu(m)|\{f\in \W_{m}^{(1)}:H(f)<H\}=
O_\epsilon\Big(\frac{H^{n(n+1)/2+\epsilon}}{M^{1-\epsilon}}\Big)+O(H^{n(n+1)/2-1}).
\end{equation}
Briefly, the proof is as follows: first, there exists a polynomial $P\in\Z[V_n]$, belonging to the algebra generated by $\Delta$ and $\Delta'$, which does not involve the constant coefficient $a_n$. Hence, we may consider $P$ as a polynomial in $a_1,\ldots, a_{n-1}$. Second, a bound of size $O(H^{n(n+1)/2-n-1})$ is easily obtained on the number of possible values of $a=(a_1,\ldots,a_{n-1})$, of bounded height, for which $P(a)=0$. Third, we fiber over the $O(H^{n(n+1)/2-n})$ values of such $a$ for which $P(a)\neq 0$. For each such $a$, it is clear that $P(a)$ has at most $O(H^\epsilon)$ different divisors. Since $P(f)=P(a_1,\ldots,a_{n-1})\equiv 0\pmod{m}$ for $f(x)=x^n+\sum_{i=1}^n a_{i}X^{n-i}\in \W_m$, fixing $a$ constrains the value of $m>M$ to be one of these $O(H^\epsilon)$ divisors of $P(a)$. Fourth and finally, we consider $\Delta$ to be a polynomial $\Delta_a(a_n)$ in $a_n$. We then note that the condition $m\mid \Delta(f)=\Delta_a(a_n)$ implies that there are at most $O((H^n/m^{1-\epsilon})+1)=O((H^n/M^{1-\epsilon})+1)$ choices for $a_n$, concluding the proof of \eqref{eqeq2temp}.

It is precisely this last step which breaks down when $m$ is not required to be squarefree. We note that $\Delta_a(a_n)$ is a polynomial of degree $n-1$ in $a_n$ whose leading coefficient $(-1)^{n(n-1)/2} n^{n-1}$ does not depend on $a$. Suppose $p^k\parallel m$ for some prime $p$ and some fixed $m\mid P(a)$. Let $\ell\ll_n 1$ be a nonnegative integer such that $g(a_n) = \Delta_a(a_n)/p^\ell\in\Z[x]$ and at least one of its coefficients is not divisible by $p$. The condition $m^2\mid \Delta(f)$ now becomes $p^{2k-\ell}\mid g(a_n)$ for every prime divisor $p$ of $m$. Let $\delta = \lceil \frac{2k-\ell}{n-1}\rceil$ and let $g(x) = f_1(x)\cdots f_j(x)$ be a factorization in $(\Z/p^\delta\Z)[x]$ where $j$ is maximal. Since the reduction $\bar{g}(x)\in\F_p[x]$ of $g(x)$ modulo $p$ is a nonzero polynomial of degree at most $n-1$, we have $\bar{g} = \bar{f_1}\cdots\bar{f_j}$ and so $j\leq n-1.$ In order for $p^{2k-\ell}\mid g(a_n)$, we then must have $p^\delta\mid f_i(a_n)$ for some $i = 1,\ldots,j$. This implies that $(x-a_n)\mid f_i(x)$ in $(\Z/p^\delta\Z)[x]$ and so by maximality of $j$, we see that $f_i(x)$ is linear. Hence the density of integers $a_n$ such that $p^{2k-\ell}\mid g(a_n)$ is at most $(n-1)/p^\delta.$ Multiplying over all prime divisors $p$ of $m$ gives that the density of integers $a_n$ such that $m^2\mid \Delta_a(a_n)$ is $O(1/m^{2/(n-1)-\epsilon})$. Combining with the proof of \cite[Theorem 3.5 and Lemma~3.6]{geosieve} recalled above gives \eqref{eqeq1}.
\end{proof}

We are ready to prove Theorem \ref{thMainPolybound}.

\medskip
\noindent\textbf{Proof of Theorem \ref{thMainPolybound}:} %
To bound the number of elements in $\cup_{m>M} \W_{m}^{(2)}$, we
recall the setup of \cite{SF1}. 
The proofs of \cite[Theorems~2.3 and 3.2]{SF1} imply that we have a map $\sigma_m:\W_{m}^{(2)}\to \frac14 W(\Z)$, injecting into the set of distinguished elements, such that the resolvent of $\sigma_m(f)$ is $f$ and $Q(\sigma_m(f))=m$. Here, $Q$ is an invariant defined on the set of distinguished elements of $W(\Z)$, given explicitly in \cite[\S2.1, \S3.1]{SF1}.
To bound the number of elements in $\cup_{m>M} \W_{m}^{(2)}$ having height less than $H$, it thus suffices to bound the number of $G(\Z)$-orbits on distinguished elements of $W(\Z)$ having height less than~$H$ and $Q$-invariant larger than $M$.
This is precisely what is carried out via geometry-of-numbers arguments in~\cite[\S\S2--3]{SF1}. Moreover, using Corollary~\ref{corhilbert} instead of the Selberg sieve in the proofs of~\cite[Propositions 2.6 and 3.5]{SF1} improves the error terms there to $O_\epsilon(H^{\dim(W)-1+1/2^{2g}+\epsilon})$. 
We thus obtain the following bound: 
\begin{equation}
\label{eqeq2}
\#\bigcup_{\substack{m>M
}}\{f\in \W_{m}^{(2)}:H(f)<H\}=
O_\epsilon(H^{n(n+1)/2+\epsilon}/M)+O_\epsilon(H^{n(n+1)/2-1+1/2^{2g}+\epsilon}).
\end{equation}
We note that \eqref{eqeq2} is a strengthening of \cite[Theorem 1.5(b)]{SF1}.
Optimizing by taking $\alpha = (n-1)/(n+3)$ and $\beta = 2/(n+3)$ in Proposition \ref{cor:sw} gives Theorem \ref{thMainPolybound}. $\Box$

\medskip

We now deduce Theorem \ref{thm:mainterm} from Theorem \ref{thMainPolybound}.

\medskip
\noindent\textbf{Proof of Theorem \ref{thm:mainterm}:} %
Applying an inclusion-exclusion sieve, we obtain
\begin{equation}\label{eq:inexmmm}
\#\bigl\{f\in V_n(\Z): H(f)<H \mbox{ and } \Delta(f) 
\mbox{ squarefree}\bigr\}=
\sum_{m\geq 1}\mu(m)\#\{f\in \W_m:H(f)<H\}.
\end{equation}
We break up the sum over $m$ into three ranges, namely, the {\it large range} consisting of $m\geq H^{n/2}$, the {\it middle range} consisting of $H\leq m< H^{n/2}$, and the {\it small range} consisting of $m<H$. We will obtain precise estimates for the sum of $m$ over the small range, and prove that the sum over~$m$ in the middle and large ranges are negligible, where we say that a number is \emph{negligible} if it is $O_\epsilon(H^{n(n+1)/2-1+1/2^{2g}+\epsilon})$.

First, note that
Theorem \ref{thMainPolybound} implies the bound
\begin{equation}\label{eq:largemmm}
\sum_{m\geq H^{n/2}}|\mu(m)|\#\{f\in \W_m:H(f)<H\}= O_\epsilon(H^{n(n+1)/2-1+1/2^{2g}+\epsilon}).
\end{equation}
Therefore, the sum over the large range is negligible.

Next we consider the middle range. That is, we sum the right hand side of \eqref{eq:inexmmm} over $m$ in the range $H<m\leq H^{n/2}$ and prove that the sum is $O_\epsilon(H^{n(n+1)/2-1+\epsilon})$. We fiber over integer tuples $a=(a_1,\ldots,a_{n-1})$. For any integer tuple $a$, let $\Delta_a(a_n)$ denote as above the discriminant of $x^n + a_1x^{n-1} + \cdots + a_n$ and let $\theta_a(m)$ denote the density of integers $a_n$ such that $m^2\mid \Delta_a(a_n).$ Let $B$ denote the set of integer tuples $a=(a_1,\ldots,a_{n-1})$ such that $|a_i|<H^i$ for $i=1,\ldots,n-1$. Then we have
\begin{equation}\label{eq:midm}
\sum_{H<m\leq H^{n/2}}|\mu(m)|\#\{f\in \W_m:H(f)<H\}= \sum_{a\in B}\sum_{H<m\leq H^{n/2}}|\mu(m)|(\theta_a(m)\cdot 2H^n + O(1)).
\end{equation}
Since $\#B = O(H^{n(n+1)/2 - n})$, we see that the sum of the $O(1)$ term is negligible.

Now the discriminant $\Delta(\Delta_a)$ of $\Delta_a(a_n)$ is a polynomial in $a_1,\ldots,a_{n-1}$. We claim that $\Delta(\Delta_a)$ has a term involving only $a_{n-1}$. Indeed, when $a_1=\cdots=a_{n-2}=0$, we have $$\Delta_a(a_n) = \Delta(x^n + a_{n-1}x + a_n) =  (-1)^{n(n-1)/2}n^na_n^{n-1}-(-1)^{n(n+1)/2}(n-1)^{n-1}a_{n-1}^n,$$
and so $$\Delta(\Delta_a) = C_na_{n-1}^{n(n-2)},$$
for some nonzero constant $C_n$ depending only on $n$. As a consequence, given any values for $a_1,\ldots,a_{n-2}$, $\Delta(\Delta_a)$ will be a nonzero polynomial in $a_{n-1}$. Hence, we have
$$\#\{a\in B\mid \Delta(\Delta_a) = 0\} = O(H^{n(n+1)/2 - n - (n-1)}).$$
Hence the contribution to the right hand side of \eqref{eq:midm} over $a\in B$ with $\Delta(\Delta_a) = 0$ is negligible.

Suppose now $a\in B$ with $\Delta(\Delta_a)\neq 0$. Take any squarefree $m$ with $H<m\leq H^{n/2}$. Let $d = \gcd(\Delta(\Delta_a),m)$ and let $m_1 = m/d$. For any prime $p\mid d$ and $p\nmid n$, the polynomial $\Delta_a(a_n)$ mod $p$ is a nonzero polynomial (since its leading coefficient is nonzero) with a repeated factor, in which case $\theta_a(p) = O(1/p).$ For any prime $p\mid m_1$ and $p\nmid n$, the polynomial $\Delta_a(a_n)$ mod $p$ is a nonzero polynomial without a repeated factor, in which case $\theta_a(p) = O(1/p^2).$ Hence, we have $\theta_a(m) = O(1/(dm_1^{2-\epsilon})),$ where we absorb any common divisors of $d$ and $n$, or of $m_1$ and $n$ into the implied constant. Denoting by $\sideset{}{'}\sum$ a sum over squarefree numbers, we have
\begin{eqnarray*}
\sum_{\substack{a\in B\\ \Delta(\Delta_a)\neq 0}} \sideset{}{'}\sum_{H<m\leq H^{n/2}} \theta_a(m) &\ll_\epsilon& \sum_{\substack{a\in B\\ \Delta(\Delta_a)\neq 0}}\Big( \sideset{}{'}\sum_{\substack{1\leq d\leq H\\ d\mid\Delta(\Delta_a)}} \sideset{}{'}\sum_{\frac{H}{d}<m_1\leq \frac{H^{n/2}}{d}}\frac{1}{dm_1^{2-\epsilon}}+\sideset{}{'}\sum_{\substack{H< d\leq H^{n/2}\\ d\mid\Delta(\Delta_a)}} \sideset{}{'}\sum_{1<m_1\leq \frac{H^{n/2}}{d}}\frac{1}{dm_1^{2-\epsilon}}\Big)\\
&\ll_\epsilon& \sum_{\substack{a\in B\\ \Delta(\Delta_a)\neq 0}} \sideset{}{'}\sum_{\substack{1\leq d\leq H^{n/2}\\ d\mid\Delta(\Delta_a)}} \frac{1}{H^{1-\epsilon}}\\
&\ll_\epsilon& H^{n(n+1)/2-n -1+\epsilon},
\end{eqnarray*}
where the last bound follows because the number of divisors of the nonzero integer $\Delta(\Delta_a)$ with each $a_i$ bounded by some fixed power of $H$ is $O_\epsilon(H^\epsilon)$. Therefore, we have proved that 
\begin{equation}\label{eq:midmmm}
\sum_{H<m\leq H^{n/2}}|\mu(m)|\#\{f\in \W_m:H(f)<H\}=O_\epsilon(H^{n(n+1)/2-1+\epsilon}).
\end{equation}

It remains to consider the small range $1\leq m \leq H$. For this range, note that $m^2\leq H^2$ is less than the range of $a_2$ and so we fiber over $a_1$ only. Denote the density of $\W_m$ in $V_n(\Z)$ by $\theta(m)$. When $m=p$ is a prime, we have $\theta(p) = O(1/p^2)$. When $m$ is squarefree in general, we have $\theta(m) = O_\epsilon(1/m^{2-\epsilon}).$ For any integer $a_1$, let $V_n(a_1,\Z)$ denote the set of monic polynomials of degree $n$ whose $x^{n-1}$-coefficient is $a_1$ and let $\theta(a_1,m)$ denote the density of $\W_m\cap V_n(a_1,\Z)$ in $V_n(a_1,\Z)$. When $m$ is coprime to $n$, we simply have $\theta(a_1,m) = \theta(m).$ 

We fiber over $a_1$ and break the regions for $a_2,\ldots,a_n$ into intervals of length $m^2$ to obtain
\begin{eqnarray*}
&&\sum_{1\leq m\leq H}\mu(m)\#\{f\in \W_m:H(f)<H\}\\
&=&\displaystyle\sum_{1\leq m\leq H}\mu(m)\sum_{|a_1|<H}(\theta(a_1,m)2^{n-1}H^{n(n+1)/2-1}+O(H^{n(n+1)/2-3}))
\\&=&2^{n-1}H^{n(n+1)/2-1}\sum_{d\mid n}\sum_{\substack{1\leq m_1\leq H/d\\ \gcd(m_1,n) = 1}} \mu(d)\mu(m_1) \sum_{|a_1| < H} \theta(a_1,d)\theta(a_1,m_1) + O(H^{n(n+1)/2-1})\\
&=& 2^{n-1}H^{n(n+1)/2-1}\sum_{\substack{1\leq m_1\leq H/d\\ \gcd(m_1,n) = 1}}\mu(m_1)\theta(m_1)\sum_{d\mid n}\mu(d)\sum_{|a_1| < H}\theta(a_1,d)+ O(H^{n(n+1)/2-1}).
\end{eqnarray*}
For any squarefree $d\mid n$, we also have
$$\sum_{|a_1|<H}(\theta(a_1,d)2^{n-1}H^{n(n+1)/2-1} + O(H^{n(n+1)/2-3})) = \theta(d)2^nH^{n(n+1)/2} + O(H^{n(n+1)/2-1}).$$
as both sides count monic polynomials of degree $n$ with height bounded by $H$ having discriminant divisible by $d^2$. Hence
$$\sum_{|a_1| < H}\theta(a_1,d) = \theta(d)\cdot 2H + O(1)$$ and thus
$$\sum_{d\mid n}\mu(d)\sum_{|a_1| < H}\theta(a_1,d) = 2H\sum_{d\mid n}\mu(d)\theta(d) + O(1),$$
since the sum of $O(1)$ over $d\mid n$ is independent of $H$. Finally, combined with
$$\sum_{\substack{1\leq m_1<H/d\\ \gcd(m_1,n) = 1}}\mu(m_1)\theta(m_1) \ll_\epsilon \sum_{m_1\geq1}\frac{1}{m_1^{2-\epsilon}} = O(1)$$ and $$\sum_{1\leq m\leq H}\mu(m)\theta(m) = \sum_{m\geq1}\mu(m)\theta(m) - \sum_{m>H}\mu(m)\theta(m) = \lambda_n - O_\epsilon(\sum_{m>H}\frac{1}{m^{2-\epsilon}}) = \lambda_n - O_\epsilon(\frac{1}{H^{1-\epsilon}}),$$ we have 
\begin{equation}\label{eq:smallmmm}
    \sum_{1\leq m\leq H}\mu(m)\#\{f\in \W_m:H(f)<H\} = \lambda_n 2^n H^{n(n+1)/2} + O_\epsilon(H^{n(n+1)/2-1+\epsilon}).
\end{equation}
The first estimate of Theorem \ref{thm:mainterm} now follows from \eqref{eq:inexmmm}, \eqref{eq:largemmm}, \eqref{eq:midmmm} and \eqref{eq:smallmmm}. The~second estimate follows similarly. $\Box$

\medskip

Next, we deduce the analogue of Theorem \ref{thMainPolybound} for polynomials with vanishing subleading coefficient.
\begin{theorem}\label{theq4}
Let $\W_m^{\circ}$ denote the set of elements in $\W_m$ that have vanishing subleading coefficient. Then
\begin{equation}\label{eqfinal}
\#\bigcup_{\substack{m>M
}}\{f\in \W^{\circ}_{m}:H(f)<H\}\ll_{n,\epsilon}
\frac{H^{(n-1)(n+2)/2+\epsilon}}{{M^{2/(n+3)-\epsilon}}}+H^{(n-1)(n+2)/2-1+1/2^{2g}+\epsilon}.
\end{equation}
\end{theorem}
\begin{proof}
For an integer $k$, the transformation $f(x)\mapsto f(x+k)$ does not change membership in $\W_m$ and changes the height of $f$ by at most  
$O(|k|)$. Hence the set of elements in $\W_m$ having vanishing subleading coefficient and height $<H$ are each equivalent (under some transformation $f(x)\mapsto f(x+k)$) to $\gg H$ elements in $\W_m$ having height $\ll H$. Therefore,
\begin{equation*}
\#\bigcup_{\substack{m>M
}}\{f\in \W^{\circ}_{m}:H(f)<H\}\ll_n
\frac{1}{H}\#\bigcup_{\substack{m>M
}}\{f\in \W_{m}:H(f)<H\}.
\end{equation*}
The result now follows from Theorem \ref{thMainPolybound}.
\end{proof}

We are now ready to prove Theorem \ref{thimp}.

\medskip
\noindent{\bf Proof of Theorem \ref{thimp}:}
Let $N_n^\prim(X)$ denote the number of primitive number fields of degree~$n$ having absolute discriminant less than $X$. As explained in the introduction, the set of primitive number fields with absolute  discriminant less than $X$ injects into the set of integer monic  polynomials of degree $n$ with vanishing subleading coefficient and height $\ll X^{1/(2n-2)}$. Let $S_X$ denote the image of this injection. Choose $\kappa=n-1$. By Proposition \ref{thm:discsmall}, it follows that away from a set $S_X'$ of size $O(X^{\frac{n+2}{4}-\frac{1}{2n-2}})$, every element in $S$ has absolute  discriminant $\gg X^{\frac{n-1}{2}}$. Since the absolute discriminant of the field corresponding to an element in $S_X$ is less than $X$ by definition, the absolute discriminant of any element in $S_X\backslash S_X'$ is divisible by $m^2$ for some $m\gg X^{\frac{n-3}{4}}$.
By Theorem~\ref{theq4}, we thus deduce that
\begin{equation*}
\begin{array}{rcl}
N_n^\prim(X)&\leq& \#S_X\;\;=\;\;\#S_X'+\#(S_X\backslash S_X')
\\[.2in]&\ll_\epsilon&\displaystyle 
X^{\frac{n+2}{4}-\frac{1}{2n-2}}+
\frac{X^{\frac{n+2}{4}+\epsilon}}{X^{\frac{n-3}{2(n+3)}}}+X^{\frac{n+2}{4}-\frac{1}{2n-2}+\frac{1}{2^{2g}(2n-2)}+\epsilon}.
\end{array}
\end{equation*}
Since $\frac{n-3}{2(n+3)}\geq\frac{1}{2n-2}$ for $n\geq 6$, we have proved the version of Theorem \ref{thimp} where $N_n(X)$ is replaced by $N_n^\prim(X)$.

Finally, we note that the bound \cite[Equation (1.2)]{Sch} with $L=\Q$ implies that the number of imprimitive number fields of degree~$n$ with absolute discriminant less than $X$ is at most  \smash{$O(X^{\frac{n}{8}+\frac{1}{2}})$}. This completes the proof of Theorem \ref{thimp}. $\Box$

\subsection*{Acknowledgments}

We are very grateful to Ashvin Swaminathan and Sameera Vemulapalli for helpful conversations and comments on an earlier version of this manuscript. We thank the referee for many helpful comments, and in particular for pointing out an issue in the previous proof of what is now Proposition \ref{propW1}. 
We also thank Theresa  Anderson, Ayla Gafni, Kevin Hughes, Robert  Lemke Oliver,
David Lowry-Duda, Frank Thorne, Jiuya Wang, and Ruixiang Zhang for sharing with us their recent preprint~\cite{AGHLLTWZ}.

The first-named author was supported by a Simons Investigator Grant and NSF Grant~DMS-1001828. The second-named author was supported by an NSERC Discovery Grant and Sloan Research Fellowship. The third-named author was supported by an NSERC Discovery Grant.

\end{document}